\documentclass[10 pt]{article}
\usepackage[header, margin]{classic}

\usepackage{xpatch}

\makeatletter
\xpatchcmd{\SL@margintext}
  {\{\SL@prlabelname{#1}\}}    
  {\SL@prlabelname{#1}}       
  {}{}
\makeatother

\usetikzlibrary{automata,arrows,positioning,calc}
\usepackage{amsrefs}

\usepackage{xcolor}

\begin{document} 

\title{Almost-optimal bulk regularity conditions in the CLT for G$\b$E}

\vspace{1cm}
\noindent

\begin{minipage}[b]{0.3\textwidth}
\hspace{3cm}
 
 \end{minipage}
 \begin{minipage}[b]{0.3\textwidth}
 \author{Albert Zhang}

\address{Courant Institute, \\
		New York University \\
acz271@nyu.edu}
 \end{minipage}

\begin{minipage}[b]{0.3\textwidth}

 \end{minipage}

\begin{abstract}
		We consider linear statistics $\sum_{k=1}^N f(x_k)$ of a Gaussian $\b$-ensemble. Following the work of \cite{LaSo22}, which was concerned  with Wigner matrices, we show that the central limit theorem holds for any $f \in H^{1/2 + \e}$ and $\b > 0$. The novel input is the Helffer-Sj\"ostrand representation for the covariance, first appearing in the random matrix context in \cite{ErYa15}, coupled with characteristic flow estimates for DBM, from \cite{Bo22}.
\end{abstract}

\tableofcontents

\newcommand{\Brown}[1]{\textcolor{Brown}{#1}}
\newcommand{\Orange}[1]{\textcolor{BurntOrange}{#1}}
\newcommand{\Navy}[1]{\textcolor{BlueViolet}{#1}}
\newcommand{\Blue}[1]{\textcolor{Cerulean}{#1}}
\newcommand{\Teal}[1]{\textcolor{PineGreen}{#1}}

\section{Introduction}%
\label{sec:intro}

\subsection{Fluctuations of linear statistics}

In this work, we are concerned with the fluctuations of linear statistics $F_N(\xx) = \sum_{k=1}^N f(x_k)$ of Gaussian $\b$-ensembles (G$\b$E),
\begin{equation}\label{d:GbE}
		d\nu(\xx) = d\nu_{N, \b}(\xx) := \frac{1}{Z_{N,\b}} \underbrace{\exp \l( - \frac{N \b}{4} |\xx|^2 + \b \sum_{i < j} \log |x_i - x_j| \r)}_{=: \exp(-\cal H(\xx))} d\xx.
\end{equation}
In particular, we consider functions $f \in H^{1/2+ \e}(\R)$ of low regularity, where we recall that the homogeneous and inhomogeneous (fractional) Sobolev spaces $\dot H^s = \dot H^s(\R)$ and $H^s = H^s(\R)$, respectively, are the spaces of all functions such that the norms
\begin{equation}
		\lV f \rV_{\dot H^s}^2 := \int |\hat f(\xi)|^2 |\xi |^{2s} d\xi, \qquad \lV f \rV_{H^s}^2 := \int |\hat f(\xi)|^2 ( 1 + |\xi|^{2s}) d\xi,
\end{equation}
are finite, where $\hat f$ denotes the Fourier transform (see (\ref{d:fourier}) for our scaling convention). Recall that for the particular case of $s = 1/2$, we have the identity,
\begin{equation}
		\lV f \rV_{\dot H^{1/2}} = \frac{1}{2\pi} \int\int_{\R^2} \frac{(f(x) - f(y))^2}{(x - y)^2} dx dy.
\end{equation}
The significance of this form in the context of $\b$-ensembles is that one expects a (centered) linear statistic $F_N$ to converge in distribution to a Gaussian random variable with variance \cites{BeLeSe18,Jo98}
\begin{equation}\label{e:V(f)}
		v(f) := \frac{1}{2\b\pi^2} \int_{-2}^2 \int_{-2}^2 \frac{(f(x) - f(y))^2}{x - y)^2} \frac{4 - xy}{\sqrt{4 - x^2} \sqrt{4 - y^2}} dx dy. 
\end{equation}
See \cites{LaSo22,LyPa09} for an analogous expression in the Wigner matrix setting, where there is an additional term depending on the fourth cumulant of the matrix entries.

The study of fluctuations and central limit theorems (CLT) for linear statistics of $\b$-ensembles was initiated by Johansson \cite{Jo98}, whose approach centered around a variational equation that is now commonly referred to as the loop equation. Later works in this direction include \cite{La21} which obtained a CLT in the regularity class $f \in C^3$ for the Circular $\b$-ensembles (C$\b$E), and \cite{BoMoPa22} which expanded the loop equation hierarchy to obtain CLTs for the Green's function and log-characteristic polynomial,
\begin{equation}
		s(z) = s_N(z) := \frac{1}{N} \sum^{N}_{k=1} \frac{1}{x_k - z}, \qquad L(E) = L_N(E) := \sum^{N}_{k=1} \log(E - x_k),
\end{equation}
for general $\b$-ensembles on the line (with standard one-cut and smoothness assumptions on the potential). On the other hand, for $f \in C^5$ and $V \in C^6$ a CLT was obtained in \cite{BeLo18} for the \emph{mesoscopic} linear statistics
\begin{equation}
		\sum_{k=1}^N f(N^a(x_k - E)),
\end{equation}
for fixed $0 < a < 1$ and $E$ in the bulk. An alternate line in \cite{BeLeSe18} obtains the CLT for $f \in C^4$ and more general assumptions (such as multi-cut) on $V$, using a transport approach. Johansson conjectured \cite{Jo98}*{Remark 2.5} that the only condition one needs for the CLT to hold should be that the limiting expression of the variance (\ref{e:V(f)}) is finite. Thus a natural line of work has been to push this conjecture for functions in the fractional Sobolev spaces $H^s$, with the aim of getting as close to $s = 1/2$ as possible. This was carried out in \cite{SoWo13} and then later in \cite{LaSo22} to obtain CLTs for smooth Wigner matrices, with a critical input being the use of Littlewood-Paley theory to decompose the test function $f$ (in $H^{1/2+\e}$ or $C^{1/2+\e}$) into scales. For a more exhaustive list of works on the fluctuations of linear statistics for other random matrix models and classes of test functions, refer to \cite{LaSo22} and the references therein.

The following is our main result. As noted in \cite{LaSo22}, the result is optimal up to the $\e > 0$ factor, as the CLT does not hold for functions in $H^{1/2 - \e}$ which can be seen from the example of the indicator function $\1_{x \leq E}$.

\begin{theorem}\label{t:var-H1/2}
		Let $\nu = \nu_{N,\b}$ be given as in (\ref{d:GbE}), for any $\b > 0$. For any $\e > 0$ and $\kappa > 0$, there is a $C = C(\b, \e, \kappa) > 0$ such that the following holds. For any $f : \R \to \R$ supported in $(-2 + \kappa, 2 - \kappa)$ we have,
		\begin{equation}\label{e:var-H1/2}
				\Var_\nu \l( \sum^{N}_{k=1} f(x_k) \r) \leq C \lV f \rV_{H^{ \frac{1}{2} + \e}}^2.
		\end{equation}
		Moreover, the centered linear statistic $F - \E[F]$ converges to a Gaussian random variable with variance as given in (\ref{e:V(f)}).
\end{theorem}

It is worth mentioning at this point that our work does not rely on the loop equation method at all\footnote{Of course, one may take the perspective that the equations arising from dynamical methods are (non-equilibrium) manifestations of the (equilibrium) loop equations, or vice versa.}, but depends rather on the more recent \emph{dynamical} breakthroughs in the understanding of Dyson Brownian motion (DBM),
\begin{equation}
		dx_k(t) = \sqrt{ \frac{2}{N\b} } dB_k(t) + \l( \frac{1}{N} \sum_{j \neq k} \frac{1}{x_k(t) - x_j(t)} - \frac{1}{2} x_k(t) \r) dt, \qquad k = 1, 2, \ldots, N,
\end{equation}
where $(B_k)_{k=1}^N$ are independent standard Brownian motions. In particular, starting from \cite{Bo22} we now have a fairly robust understanding of the homogenization of the nonlocal dynamics,
\begin{equation}\label{e:nonlocal-dyn}
		\frac{d}{dt} \g u_k(t) = \l( \l[\V^2 \cal H(\xx_t) \r] \g u \r) (k) = \frac{1}{N}\sum_{j \neq k} \frac{\g u_j(t) - \g u_k(t)}{( x_k(t) - x_j(t))^2},
\end{equation}
which can be seen as a noisy $\D^{1/2}$ operator driven by a log-gas random environment.  An analysis by characteristics and maximum principle yield the (submicroscopic) estimate  (\ref{e:gap-coupling-N2t}), which enabled the authors in \cite{LaSo22} to push previous techniques (notably, Littlewood-Paley theory) in \cite{SoWo13}, extending their CLT to more general smooth Wigner matrices and $f \in H^{1/2 + \e}$. A large part of our work builds off of the multiscale Littlewood-Paley analysis\footnote{In fact, most of the analysis is exactly the same, by just replacing the words ``smooth Wigner matrix'' with ``Gaussian $\b$-ensemble''.} in \cite{LaSo22}, with our main contribution being a robust way to handle the mesoscopic scales $(N^{-1}, 1]$, which now works for general $\b > 0$. In particular, we recall a certain Helffer-Sj\"ostrand representation\footnote{Not to be confused with the \emph{other} Helffer-Sj\"ostrand representation in the random matrix literature which is built from a quasi-analytic extension of $f$. Though, the two certainly can be seen to be part of the same (methodological) forest, even if they are not the same (technical) tree.}, which has its (mathematical physics) roots in \cites{HeSj94,NaSp97}, and later saw an application to (random matrix and $\b$-ensembles) universality in \cite{ErYa15}. However, until (relatively recent) sharp estimates for the heat kernel of the dynamics (\ref{e:nonlocal-dyn}) were developed \cite{BoCi26}, one could not leverage the Helffer-Sj\"ostrand representation fully towards the CLT program. This is the (now feasible) step that we carry out in this paper.

\subsection{Organization and notation of the paper}

In this paper, we denote $c, C > 0$ to be generic small and large constants, respectively, which do not depend on $N$ but may vary from line to line. Also let
\begin{equation}
		\vphi = e^{C_0 (\log \log N)^2}
\end{equation}
be a subpolynomial error parameter for some fixed $C_0 > 0$. We denote $s(z) = s_N(z) = \frac{1}{N} \sum_{k=1}^N \frac{1}{x_k - z}$ the Green's function associated to the particles of the G$\b$E (\ref{d:GbE}), and
\begin{equation}
		\rho_{\t{sc}}(x) = \frac{1}{2\pi} \sqrt{4 - x^2} \1_{|x| < 2}, \qquad m_{\t{sc}}(z) = \frac{-z + \sqrt{z^2 - 4}}{2},
\end{equation}
the semicircle law and its Stieltjes transform, respectively. Define the $k$-th typical location $\ga_k$ through $\int_{-\infty}^{\ga_k} d\rho_{\t{sc}} = \frac{k}{N}$. For $N_1 < N_2 \in \N$ we use the notation $[\![N_1, N_2]\!]  \{N_1, N_1 + 1, N_1 + 2, \ldots, N_2\}$ and $[\![N ]\!] = [\![ 1, N]\!]$. We say that an index (or particle) $k \in [\![N]\!]$ is a bulk index (or bulk particle) if there is some $0 < \a < 1/2$ (or $0 < \d < 2$) such that $k \in [\![\a N, (1 - \a) N]\!]$ (or $x_k \in (-2 + \d, 2 - \d)$, respectively).

Regarding asymptotics, we use $x \lesssim y$ to denote that there is a universal constant $C > 0$ such that $x \leq C y$ for all sufficiently large $N$. If both $x \lesssim y$ and $y \lesssim x$ then we may use $x \sim y$ or $x\asymp y$. We say that an event $\cal E$ occurs with \emph{overwhelming probability} if for any $D > 0$ (large, but not depending on $N$), there exists an $N_0 > 0$ such that for all $N > N_0$ it holds
\begin{equation}
		\P(\cal E) > 1 - N^{-D}.
\end{equation}
In this paper, we take for granted all existing local laws and their consequences, such as rigidity along Dyson Brownian motion (DBM), e.g. \cite{Bo22}*{Lemma 2.3}.

Regarding organization, we recall and explain the Helffer-Sj\"ostrand representation in the context of DBM in section \S\ref{sub:Decomposition along Langevin}. In \S\ref{sub:cov-repr} we use this representation to obtain an estimate on the two-point correlators $\Cov_\nu(s(z), s(w))$. In \S\ref{sec:clt-lin-stat} we carry out the same program as in \cite{LaSo22}, briefly recalling the needed Littlewood-Paley theory in \S\ref{sub:Littlewood-Paley decompositions},and then handling the mesoscopic, submicroscopic, and very high frequencies regimes in \S\ref{sub:var-lin-stat}, \S\ref{sub:regime2}, and \S\ref{sub:regime3}, respectively. The proof of the main result, Theorem \ref{t:var-H1/2}, also is organized and done in \S\ref{sub:var-lin-stat}.

\subsection{AI use statement, acknowledgements}

The model ChatGPT 5.6 was used for mathematical exploration and literature searches pertaining (for example) to the keywords: loop equations, CLT for linear statistics, Poincar\'e inequalities. The only mathematical arguments found by the LLM are the short technical Lemmas \ref{l:wegner} and \ref{l:log-div-1/Z} in the appendix. The main idea to use the Helffer-Sj\"ostrand representation (originally applied to a random matrix context in \cite{ErYa15}, for covariances between gaps observables) to handle mesoscopic scales is the author's own, after Paul Bourgade suggested to focus our approach on the two-point correlators,
\begin{equation}
		\Cov_\nu(s(z), s(w)), \qquad z, w \in \C \setminus \R .
\end{equation}
No AI model was used in the writing process of this paper. The author would like to thank Paul Bourgade for many useful conversations.

%

\section{Helffer-Sj\"ostrand representation}%
\label{sec:HS-repr}

\subsection{Decomposition and homogenization along Dyson Brownian motion}%
\label{sub:Decomposition along Langevin}
Recall the Gaussian $\b$-ensemble, given by
\begin{equation}
		d\nu(\xx) = d\nu_{N, \b} := \frac{1}{Z_{N,\b}} \underbrace{\exp \l( - \frac{N \b}{4} |\xx|^2 + \b \sum_{i < j} \log |x_i - x_j| \r)}_{=: \exp(-\cal H(\xx))} d\xx.
\end{equation}
It is the equilbrium measure for the following system of SDE's,
\begin{equation}\label{e:dbm-sde}
		dx_k(t) = \sqrt{ \frac{2}{N\b} } dB_k(t) + \l( \frac{1}{N} \sum_{j \neq k} \frac{1}{x_k(t) - x_j(t)} - \frac{1}{2} x_k(t) \r) dt, \qquad k = 1, 2, \ldots, N,
\end{equation}
where $(B_k)_{k=1}^N$ are independent standard Brownian motions. The generator associated to the above (reversible) Langevin dynamics is given by
\begin{equation}
		L f = L_{N,\b} f = \frac{1}{N\b} \l[ \D f - \V \cal H \* \V f \r] ,
\end{equation}
Let $(Q_t)_{t \geq 0}$ denote the associated Markov semigroup, self-adjoint with respect to $d\nu$. 

The basic idea of the Helffer-Sj\"ostrand representation is to leverage the following commutation relation:
\begin{equation}\label{e:VL=LV}
		\V L = \bbL \V,
\end{equation}
where the operator $\bbL$ acts on vector fields $F = (F^i)_{i=1}^N$ through (with Einstein notation)
\begin{equation}
		\l[ \bbL F \r]^i = L F^i - \frac{1}{N\b}  \l( \V^i \V_k \cal H \r) \l( \V^k F \r).
\end{equation}
Visually, 
\begin{equation}
		\bbL F = \l( \begin{bmatrix}
				L & 0 & \cdots & 0 \\
				0 & L & \ddots & \vdots \\
				\vdots & \ddots & L & 0 \\
				0 & \cdots & 0 & L
		\end{bmatrix} - \frac{1}{N\b}\V^2 \cal H \r) \begin{bmatrix}
				F^1 \\
				\vdots \\
				F^N
		\end{bmatrix}.
\end{equation}
Notably, from (\ref{e:VL=LV}) we have that
\begin{equation}\label{e:Gtx-evo}
		\bd_t \V Q_t f = \bbL \V Q_t f,
\end{equation}
from which we may deduce the following Feynman-Kac representation\footnote{This is first derived in the $\b$-ensembles context in \cite{ErYa15}, with a different scaling in $(N,\b)$ convention.} for $\GG(t, \xx) := \V P_t f(\xx)$:
\begin{proposition}\label{p:feyn-kac}
		Let $\E_{\xx}$ denote expectation with respect to the process (\ref{e:dbm-sde}) started at $\xx_0 = \xx$. For all $t > 0$, 
		\begin{equation}\label{e:Gtx-FK}
				\GG(t, \xx) = \E_{\xx} \l[ \exp \l( - \frac{1}{N\b} \int_0^t \V^2 \cal H(\xx_s) ds \r) \GG(0, \xx(t)) \r].
		\end{equation}
\end{proposition}

Now, for any fixed path $\{\xx(s) : s > 0\}$, the matrix exponential in (\ref{e:Gtx-FK}),
\begin{equation}
		\cal U(t) = \cal U(t; \xx(\*)) := \exp \l( - \frac{1}{N\b} \int_0^t \V^2 \cal H(\xx_s) ds \r),
\end{equation}
satisfies the matrix evolution equation,
\begin{equation}
		\bd_t \cal U(t) = - \frac{1}{N\b} \cal U(t) \V^2 \cal H(\xx_t), \qquad \cal U(0) = \id_{N \x N}.
\end{equation}
Each of the rows of $\cal U(t)$, denoted $\g u^a = (\g u^a_b)_{b=1}^N = (\cal U_{ab})_{b=1}^N$, satisfies the non-local parabolic equation (subject to the random environment/path $\xx(\*)$),
\begin{equation}\label{e:nonlocal}
		\bd_t \g u_k(t) = \frac{1}{N\b} \sum_{\ell \neq k} \frac{\g u_\ell(t) - \g u_k(t)}{(x_k(t) - x_\ell(t))^2} =: \l( \cal B(t) \g u(t) \r)(k).
\end{equation}
The operator $\cal B(t) = -\frac{1}{N\b} \V^2 \cal H(\xx_t)$ is now well-understood. Beginning in \cite{Bo22}, it was observed that observables of type
\begin{equation}
		f_t(z) := e^{-t/2} \sum^{N}_{k=1} \frac{\g u_k(t)}{x_k(t) - z}
\end{equation}
can be used to control solutions $\g u_k(t)$ to the nonlocal equation (\ref{e:nonlocal}), via the method of (stochastic) characteristics. Explicitly, the dynamics of $f$ are given, to leading order, by the complex Burgers equation
\begin{equation}
		\bd_t h = \frac{\sqrt{z^2 - 4}}{2} \bd_z h,
\end{equation}
whose characteristics are 
\begin{equation}\label{e:zt}
		z_t = \frac{e^{t/2} \l( z + \sqrt{z^2 - 4} \r) + e^{-t/2} \l( z - \sqrt{z^2 - 4} \r)}{2}.
\end{equation}
That is, we have the approximation $f_t(z) \approx f_0(z_t)$, which is made precise by the following lemma, which will appear in \cite{BoCi26}:

\begin{lemma}\label{l:ch-ftz}
		Let $z = E + \i \h$ with $E \in [-2 + \e, 2 - \e]$ and $\h > 0$. Then with overwhelming probability, we have 
		\begin{equation}
				|f_s(z) - f_0(z_s)| \leq \vphi^C \frac{\Im f_0(z_s)}{\sqrt{N\h \Im m_0(z_s)}} \qquad \t{for all} \qquad  0 < s < 1.
		\end{equation}
\end{lemma}


%

While the estimate along characteristics, Lemma \ref{l:ch-ftz}, will suffice for mesoscopic scales, we will need some control in the submicroscopic regime $N^{-2} \lesssim \h \lesssim N^{-1}$. At the time of writing, it seems that the only way to observe anything below the microscopic scale stems from \cite{Bo22}*{Theorem 3.1}. To state the needed result, consider two copies $(\xx(t))_{t \geq 0}$ and $(\yy(t))_{t \geq 0}$ of the DBM process (\ref{e:dbm-sde}), both driven by the same Brownian motion; their initial conditions will be specified based on context. Define,
\begin{align}
		X(z, t) &:= \frac{1}{N\Im m_{\t{sc}}(z_t)} \l[ \Im \sum_{k=1}^N \log \l( x_k(0) - z_t \r) - N \int \Im \log (x - z_t) \rho_{\t{sc}}(x) dx \r], \\
		Y(z, t) &:= \frac{1}{N\Im m_{\t{sc}}(z_t)} \l[ \Im \sum_{k=1}^N \log \l( y_k(0) - z_t \r) - N \int \Im \log (x - z_t) \rho_{\t{sc}}(x) dx \r].
\end{align}
The following \cite{LaSo22}*{Theorem 2.6} is a reformulation of \cite{Bo22}*{Theorem 3.1}.
\begin{theorem}\label{t:LaSo22-2.6}
		Let $(\xx(t))_{t \geq 0}$ and $(\yy(t))_{t \geq 0}$ each be as above in (\ref{e:dbm-sde}), with independent initial conditions and driven by the same Brownian motion. Fix an $\e, \d, \kappa > 0$. For any $D > 0$ there exists a constant $C > 0$  and $N_0 = N_0(D) > 0$ so that for any $k \in [\![ N ]\!]$ and $E$ such that $E, \ga_k \in (-2 + \kappa, 2 - \kappa)$, it holds
		\begin{equation}\label{e:gap-coupling-N2t}
				\P \l( \l|\l[  x_k(t) - y_k(t) \r]  - \l[ X(E, t) - Y(E, t)\r] \r| > N^\e \frac{1 + N |E - \ga_k|}{N^2 t} \r) \leq C N^{-D},
		\end{equation}
		for every $N > N_0$ and $N^{-1 + \d} < t < 1$.
\end{theorem}

\subsection{Representation for the covariance}%
\label{sub:cov-repr}

Due to ergodicity,
\begin{equation}
		\lim_{t \to \infty} Q_t g = \E_\nu \l[ g \r],
\end{equation}
so that we may decompose the covariance along the dynamics (\ref{e:dbm-sde}),
\begin{equation}
		\Cov_\nu (f, g) = -\int_0^\infty \frac{d}{dt} \E_\nu \l[ f Q_t g \r] = \frac{1}{N\b} \int_0^\infty \E_\nu \l[ \< \V f, \V Q_t g\> \r].
\end{equation}
Then, the Proposition \ref{p:feyn-kac} leads to the following Helffer-Sj\"ostrand representation:
\begin{equation}
		\Cov_\nu(f, g) = \frac{1}{N\b} \int_0^\infty \E_\nu \l[ \V f(\xx) \* \E_{\xx} \l[ \cal U(t) \V g(\xx_t) \r] \r] dt.
\end{equation}
It is this formula which forms the basis of our work.

\begin{proposition}[Representation for the two-point resolvent correlator]
		Let $z, w \in \C \- \R$. Then for any $T > 0$, it holds
		\begin{equation}\label{e:c2-repr}
				\Cov_\nu(s(z), s(w)) = \frac{1}{N^3\b} \int_0^T \sum_{a, b = 1}^N \E_\nu \l[ \frac{1}{(x_a - z)^2} \E_{\xx} \l[ \frac{\g u_b^a(t)}{(x_b(t) - w)^2}  \r]  \r] dt + \Cov_\nu \l( s(z), Q_T [s(w)]  \r).
		\end{equation}
\end{proposition}

\begin{lemma}[Estimate on the covariance]
		Fix a $\kappa > 0$. Let $z, w \in \C \- \R$ and denote $\z = E_\z + \i\h_\z$  for $\z \in \{z, w \}$. Assume $E_z, E_w \in (-2 + \kappa, 2 - \kappa)$ and $\h_z \asymp \h_w \gtrsim N^{-1}$. Then for any $T \lesssim N^{-1}$, it holds
		\begin{equation}\label{e:c2-est}
				\l| \Cov_\nu \l( s(z), s(w) \r) \r| \lesssim  \vphi^C \frac{T}{N^2\b \h_z \h_w \sqrt{N \h_w}} \* \frac{\h_z \w \h_{w_t}}{(E_z - E_w)^2 + (\h_z \w \h_{w_t})^2} + \frac{e^{-cN\b T}}{N^2 \b \h_z \h_w}.
		\end{equation}
		The same estimate holds if we replace either of $s(z)$ or $s(w)$ with their imaginary or real parts.
\end{lemma}

\begin{proof}
		We start by controlling the short-time component in (\ref{e:c2-repr}). Let $0 < t < T$ and write
		\begin{align*}
				e^{-t/2} \l| \sum^{N}_{a,b=1} \E_\nu \l[ \frac{1}{(x_a - z)^2} \E_{\xx} \l[ \frac{\g u_b^a(t)}{(x_b(t) - w)^2} \r] \r] \r| &\leq  \E_{\nu} \l[ \frac{1}{\h_z}\Im \sum^{N}_{a=1}  \frac{1}{x_a - z} \E_{\xx} \l[ \frac{e^{-t/2}}{\h_w}\Im \sum^{N}_{b=1} \frac{\g u_b^a(t)}{x_b(t) - w} \r] \r] \\
																																			   &= \frac{1}{\h_z \h_w} \E_\nu \l[ \Im \sum^{N}_{a=1}  \frac{1}{x_a - z} \E_{\xx} \l[ \Im f_t \l[ \g u^a \r](w) \r] \r] \\
																																			   &\leq \frac{1}{\h_z \h_w} \E_\nu \l[ \Im \sum^{N}_{a=1}  \frac{1}{x_a - z} \E_{\xx} \l[ \frac{\vphi^C}{\sqrt{N \h_w}}\Im f_0 \l[ \g u^a \r](w_t) \r] \r] \\
																																			   &= \frac{\vphi^C}{ \h_z \h_w \sqrt{N \h_w}} \E_\nu \l[ \Im \sum^{N}_{a=1}  \frac{1}{x_a - z} \Im \frac{1}{x_a - w_t} \r],
		\end{align*}
		where we applied Lemma \ref{l:ch-ftz} to get the second inequality. To estimate the expectation, we may apply rigidity to approximate by the convolution of two Cauchy kernels. Precisely, on the rigidity event we have the deterministic estimate (here we need the assumptions $\min(\h_z, \h_w) \gtrsim N^{-1}$, and $T \lesssim N^{-1}$ so that $\h_{w_t} \asymp \h_w \asymp \h_z$),
		\begin{multline*}
				\sum^{N}_{a=1} \frac{\h_z}{(x_a - E_z)^2 + \h_z^2} \frac{\h_{w_t}}{(x_a - E_{w_t})^2 + \h_{w_t}^2} \lesssim \sum_{a : \dist(a, \{E_z, E_{w_t}\}) \geq \max(\h_z, \h_{w_t})} \frac{\h_z}{(\ga_a - E_z)^2 + \h_z^2} \frac{\h_{w_t}}{(\ga_a - E_{w_t})^2 + \h_{w_t}^2} \\
				+ \sum_{a : |x_a - E_z| < \h_z} \frac{1}{\h_z} \frac{\h_{w_t}}{(E_z - E_{w_t})^2 + \h_{w_t}^2} + \sum_{a : |x_a - E_{w_t}| < \h_{w_t}} \frac{\h_z}{(E_z - E_{w_t})^2 + \h_z^2} \frac{1}{\h_{w_t}} \\
				\lesssim N \underbrace{\int \frac{\h_z}{(x - E_z)^2 + \h_z^2} \frac{\h_{w_t}}{(x - E_{w_t})^2 + \h_{w_t}^2} dx}_{\rho_{\t{Cauchy}(E_z, \h_z + \h_{w_t})}(E_w)} + N \frac{\h_z \w \h_{w_t}}{(E_z - E_w)^2 + (\h_z \w \h_{w_t})^2}.
		\end{multline*}
		Off the rigidity event the contribution is killed by the overwhelmingly small probability $N^{-D}$. Altogether, for $T = o(1)$ we have
		\begin{equation}
				\frac{1}{N^3\b} \int_0^T \sum^{N}_{a,b=1}  \E_{\nu} \l[ \frac{1}{(x_a - z)^2} \E_{\xx} \l[ \frac{\g u_b^a(t)}{(x_b(t) - w)^2} \r] \r] \lesssim \vphi^C \frac{T}{N^2\b \h_z \h_w \sqrt{N \h_w}} \* \frac{\h_z \w \h_{w_t}}{(E_z - E_w)^2 + (\h_z \w \h_{w_t})^2}.
		\end{equation}
		As for the long-time part of (\ref{e:c2-repr}), we may apply a (crude) spectral gap estimate together with the second-moment local law estimate \cite{BoMoPa22}*{Theorem 1.1}, for some $c > 0$,
		\begin{equation}
				\Cov_\nu(s(z), P_T \l[ s(w) \r]) \leq \Var_\nu (s(z))^{1/2} \Var_\nu \l( Q_T \l[ s(w) \r] \r)^{1/2} \leq \frac{1}{N^2 \h_z \h_w} e^{-cN\b T}.
		\end{equation}
		Thus, for any $T > \vphi/N$, this remainder term is negligible. For the two-point correlator between the imaginary parts $\Im s$, the proof is the same  once we bound
		\begin{equation}
				\l| \bd_{x} \Im \frac{1}{x - z} \r| \leq \l| \Im \frac{1}{(x - z)^2} \r| \leq \frac{1}{|x - z|^2},
		\end{equation}
		and apply the steps above. The real part is similar.
\end{proof}

\section{CLT for linear statistics}%
\label{sec:clt-lin-stat}

\subsection{Littlewood-Paley decompositions}%
\label{sub:Littlewood-Paley decompositions}

In this section we recall some facts about Littlewood-Paley decompositions, which were applied to random matrix theory in the works \cites{SoWo13, LaSo22} in order to extend the CLT for linear statistics to the regularity class $H^{ \frac{1}{2} + \e}$. Our treatment is essentially the same as the ones given in those works; for proofs of the following statements see the references listed therein. We denote the Fourier transform of $f: \R \to \C$ by
\begin{equation}\label{d:fourier}
		\hat f(\xi) := \int_\R e^{\i \xi x} f(x) dx.
\end{equation}
Given $f \in \cal S$ where $\cal S$ is the space of Schwarz distributions, the inhomogeneous Littlewood-Paley decomposition is the sum,
\begin{equation}
		f = \sum_{\ell = -1}^\infty f_\ell, \quad \t{where } \begin{cases}
				f_{-1} = h * f,& \\
				f_\ell = 2^\ell \om(2^\ell \*) * f & \quad \ell \geq 0,
		\end{cases}	
\end{equation}
where $h$ and $\om$ are smooth functions satisfying the following. The Fourier transforms $\hat h$ and $\hat \om$ are even, smooth, and of compact support. In particular, 
\begin{equation}
		\supp \hat h \subset \{|\xi| \leq 4/3\}, \qquad \supp \hat \om \subset \{3/4 \leq |\xi| \leq 8/3\}.
\end{equation}
Furthermore, $0 \leq \hat h, \hat \om \leq 1$ and it holds,
\begin{equation}
		\hat h(\xi) + \sum_{\ell \geq 0} \hat \om(2^{-\ell} \xi) = 1.
\end{equation}
The Littlewood-Paley decomposition allows for a convenient decomposition of $f$ into convolutions with Poisson kernels at dyadic scales. Namely, recall that the Poisson (or Cauchy) kernel is given by 
\begin{equation}
		P_\h(x) := \frac{1}{\pi} \frac{\h}{x^2 + \h^2} = \frac{1}{\pi}\Im \frac{1}{x - \i \h},
\end{equation}
and its Fourier transform is $\hat P_\h(\xi) = e^{-\h |\xi|}$. Consequently we define,
\begin{equation}
		\hat g_\ell(\xi) = e^{2^{-\ell}|\xi|} \hat f_\ell(\xi).
\end{equation}
The following theorem is from \cite{SoWo13}.
\begin{theorem}\label{t:Hs-L2}
		Let $s \geq 0$ and suppose $f \in H^s$. Then $f$ has the representation
		\begin{equation}\label{e:f-sum-Pg}
				f(x) = \sum^{\infty}_{\ell = -1}  (P_{2^{-\ell}} * g_\ell)(x),
		\end{equation}
		where the sum on the right is convergent in $H^s$. Then there is a constant $C = C(s)$ such that
		\begin{equation}
				\frac{1}{C_s} \lV f \rV_{H^s}^2 \leq \sum^{\infty}_{\ell = -1} 2^{2\ell s} \lV g_\ell \rV_{L^2}^2 \leq C_s \lV f \rV_{H^s}^2.
		\end{equation}
\end{theorem}

It will also be useful for us to record a decay estimate on the components, \cite{LaSo22}*{Proposition 2.9}.

\begin{proposition}\label{p:LP-decay}
		Let $f \in H^s$ be supported in $(-2 + \kappa, 2 - \kappa)$. Let $|x| > 2 - \kappa$. Then, for every $M > 0$ there is a $C_M > 0$ so that for all $\ell \geq 0$,
		\begin{equation}
				|f_\ell(x)| \leq C_M \frac{2^\ell}{1 + 2^{\ell M} \l| |x| - (2 - \kappa) \r|^M} \lV f \rV_{L^2}.
		\end{equation}
		The same estimate holds for $g_\ell(x)$.
\end{proposition}

\subsection{Bound on the variance of linear statistics}%
\label{sub:var-lin-stat}

In this section, we prove the main result, Theorem \ref{t:var-H1/2}, which we recall says that for any $\e, \kappa > 0$ and bulk observable $f: \R \to \R$ supported in $(-2 + \kappa, 2 - \kappa)$, there is a $C = C(\e) > 0$ such that
\begin{equation}
		\Var_\nu \l( \sum^{N}_{k=1} f(x_k) \r) \leq C \lV f \rV_{H^{ \frac{1}{2} + \e}}^2.
\end{equation}
Note that from the Littlewood-Paley decomposition Theorem \ref{t:Hs-L2} and Cauchy-Schwarz, it actually suffices\footnote{In \cite{LaSo22} it is explained why we may consider only logarithmically many $g_\ell$ in a range $N^{1/A} \leq 2^\ell \leq N^A$. Basically, for very small frequencies $2^{\ell} < N^{1/A}$ we have a function $f_\ell$ whose $C^5$ norm is sufficiently small, and so may be handled by existing techniques (e.g. loop equations). On the other hand, the trivial estimate $\Var(\sum_{k=1}^N f(x_k)) \leq N^2 \lV f \rV_\infty^2$ and the embedding of $H^{1/2 + \g a}$ into $L^\infty$ are enough to handle extremely large frequencies $2^{\ell} > N^A$ (say $A > 10^5$).}, for some $A > 1$,  to prove for each $A^{-1} \log N \lesssim \ell \lesssim A \log N$ the bound 
\begin{equation}\label{e:var-gl}
		\Var_\nu \l( \sum^{N}_{k=1} \underbrace{(P_{2^{-\ell}} * g_\ell)}_{f_\ell}(x_k) \r) \leq C N^{\e} 2^{ \ell} \lV g_\ell \rV_{L^2}^2
\end{equation}
for some $C = C(\e)$ not depending on $\ell$. We split our analysis into three cases:
\begin{enumerate}[(1)]
		\ii The mesoscopic (and microscopic) regimes $N^{-1} \lesssim 2^{-\ell} \lesssim N^{-c}$;
		\ii The submicroscopic regime $N^{-2+ \d_1} \lesssim 2^{-\ell} \lesssim N^{-1}$;
		\ii The very high frequency regime $N^{-C} \lesssim 2^{-\ell} \lesssim N^{-2 + \d_1}$.
\end{enumerate}
Above, the constant $\d_1 > 0$ is chosen sufficiently small depending on the $\g a > 0$ such that $f \in H^{1/2 + \g a}$. We will treat regimes (2) and (3) in the same way as in \cite{LaSo22}. The main novelty of our paper is that regime (1) is made simpler from our Helffer-Sj\"ostrand two-point estimate (\ref{e:c2-est}). 

\begin{remark}
		If we could improve our two-point estimate (\ref{e:c2-est}) to 
		\begin{equation}\label{conj:c2-est}
				N^2 \l| \Cov_\nu(s(z), s(w)) \r| \lesssim \vphi^C \frac{1}{(E_z - E_w)^2 + (\h_z \w \h_w)^2},
		\end{equation}
		i.e. remove all deteriorating factors of $ \frac{1}{N\h}$, then the estimate on the covariance would actually suffice in all three regimes (1)-(3). From \cite{LaSo22}*{Proposition 8.4} the improved estimate (\ref{conj:c2-est}) indeed holds for the GUE (up to an $N^\e$ factor). However, it is unclear if it is true for general $\b$ or potential $V$.
\end{remark}

%

We will handle regimes (2) and (3) in \S\ref{sub:regime2} and \S\ref{sub:regime3}, respectively. Assuming these, we may now complete the proof of our main theorem, i.e. the following proof handles the remaining mesoscopic regime (1):

\begin{proofof}{Theorem \ref{t:var-H1/2}}
		Expanding, we have
		\begin{align}
				\Var_\nu \l( \sum^{N}_{k=1} (P_{2^{-\ell}} * g_\ell)(x_k) \r) &= \frac{N^2}{\pi^2} \Var_\nu \l(   \int  g_\ell(E) \Im s(E + \i 2^{-\ell}) dE  \r) \\
																			  &= \frac{N^2}{\pi^2} \int_\R \int_\R g_\ell(E_z) g_\ell(E_w) \Cov_\nu\l(\Im s(E_z + \i 2^{-\ell}), \Im s(E_w + \i 2^{-\ell})\r) dE_z dE_w. \label{pf:var-pl}
		\end{align}
		Applying Young's inequality and our two-point estimate (\ref{e:c2-est}) gives the bulk estimate
		\begin{align*}
				&\frac{N^2}{\pi^2} \int_{-2+\kappa/2}^{2-\kappa/2} \int_{-2+\kappa/2}^{2 - \kappa/2} g_\ell(E_z) g_\ell(E_w) \Cov_\nu\l(\Im s(E_z + \i 2^{-\ell}), \Im s(E_w + \i 2^{-\ell})\r) dE_z dE_w \\
				&\lesssim \vphi^C \int_{-2+\kappa/2}^{2-\kappa/2} \int_{-2+\kappa/2}^{2-\kappa/2} g_\ell(E_z)^2 \frac{ \frac{1}{(N 2^{-\ell})^{3/2}}}{(E_z - E_w)^2 + (2^{-\ell})^2} dE_z dE_w \lesssim \vphi^C \frac{2^{\ell}}{(N 2^{-\ell})^{3/2}} \lV g_\ell \rV_{L^2}^2.
		\end{align*}
		For the edge contributions, we use Proposition \ref{p:LP-decay} to estimate, by Cauchy-Schwarz and the crude bound $|\Im s(E + \i 2^{-\ell})| \leq 2^\ell$,
		\begin{align*}
				&\frac{N^2}{\pi^2} \int_{[-2+\kappa/2,2-\kappa/2]^c}\int_{[-2+\kappa/2,2-\kappa/2]^c} g_\ell(E_z) g_\ell(E_w) \Cov_\nu\l(\Im s(E_z + \i 2^{-\ell}), \Im s(E_w + \i 2^{-\ell})\r) dE_z dE_w \\
				&\lesssim C_M^2 N^2 \lV f \rV_{L^2}^2 \int_{2-\kappa/2}^\infty \int_{2 - \kappa/2}^\infty \l( \frac{2^\ell}{1 + 2^{\ell M} \l| |E_z| - (2 - \kappa)\r|^M } \r) \\
				&\qquad\qquad\qquad\qquad\qquad\qquad \x \l( \frac{2^\ell}{1 + 2^{\ell M} \l| |E_w| - (2 - \kappa)\r|^M } \r) \Var_\nu \l( \Im s(E_z + \i 2^{-\ell}) \r) dE_z dE_w\\
				&\lesssim C_M^2 N^2 2^{4\ell}  \lV f \rV_{L^2}^2 \l( \int_{2-\kappa/2}^\infty \frac{1}{1 + 2^{\ell M} \l| |E| - (2 - \kappa) \r|^M} dE \r)^2 \\
				&\lesssim C_M' N^2 2^{-2\ell (M - 2)} \lV f \rV_{L^2}^2.
		\end{align*}
		Since $\ell \gtrsim \log_2(N)$, we may choose $M > 100$ (say) to make this term negligible. Thus, for scales down to the microscopic scale $2^{\ell} \gtrsim 1/N$, the bound (\ref{e:var-gl}) holds. 

		Now let us consider submicroscopic scales $N^{-2 + \d_1} \lesssim 2^{-\ell} \lesssim N^{-1}$. From Propositions \ref{p:submicro-close} \& \ref{p:submicro-edge} it is clear that we may restrict only to those pairs of bulk indices which are far, i.e. $(a, b) \in I_\a^2$ with
		\begin{equation}
				|a - b| > \frac{N^{\d_1}}{N\h}.
		\end{equation}
		For these indices, we first apply Lemma \ref{l:cov-stab-submicro} to replace our particles $x_a$ with randomly shifted versions $\tilde \Xi_a = \Xi_a + Z_a$, with small error term
		\begin{equation}
				\frac{N^{3\e + \d_2}}{\h} \lV g_\ell \rV_2^2,
		\end{equation}
		followed by Lemma \ref{l:cov-Xi-tiny} to bound the associated covariances of the $\tilde\Xi_a$'s, which have an even smaller error term $N^{-100} \lV g_\ell \rV_2^2$ (at scales $\h \gg N^{-2}$). 

		Finally, for even smaller scales $N^{-A} \lesssim \h \lesssim N^{-2 + \d_1}$ we may simply apply Propositions \ref{p:submicro-close} \& \ref{p:submicro-edge} for which $ \frac{N^{\d_1}}{N\h} > N$ so the sum now covers all pairs $(a,b) \in [\![N ]\!]$. Thus, for all submicroscopic scales we have
		\begin{equation}
				\Var \l( \sum_{a=1}^N f_\ell(x_a) \r) \lesssim N^{3\e + \d_1 + \d_2} 2^{\ell} \lV f_\ell \rV_2^2 + 2^{-M\ell} N^2 \lV f \rV_2^2,
		\end{equation}
		and since $\ell \geq \log_2(N)$ we may choose $M > 100$ (say) to absorb the $N^2$ into the $2^{-M\ell}$. Summing over $\ell$ and applying (\ref{e:f-sum-Pg}) yields (\ref{e:var-H1/2}).
\end{proofof}

\subsection{Submicroscopic scales}%
\label{sub:regime2}

In this section, we fix a $\d_1 > 0$ and $\h = 2^{-\ell}$ satisfying
\[
		\frac{N^{\d_1}}{N^2} \lesssim \h \lesssim \frac{1}{N}.
\]
The main idea, as in \cite{LaSo22}, is to realize each $x_a$ via a decomposition
\begin{equation}\label{a:xa-decomp}
		x_a \approx \Xi_a + Z_a,
\end{equation}
where $(Z_a, Z_b)$ is some (jointly) Gaussian noise independent of $(\Xi_a, \Xi_b)$ for each pair of bulk indices $(a, b) \in I_\a^2$. Then, by Fourier representation, 
\begin{align}
		\E_\nu \l[ f_\ell(x_a) f_\ell(x_b) \r] &\approx \E_\nu \l[ \int \int \hat f_\ell(\xi_1) \hat f_\ell(\xi_2) e^{\i \xi_1 \l(\Xi_a + Z_a\r) + \i \xi_2 \l( \Xi_b + Z_b \r)} d\xi_1 d\xi_2 \r] \label{a:xa-decomp-Cov} \\
											   &= \int \int \hat f_\ell(\xi_1) \hat f_\ell(\xi_2) \E_\nu \l[ e^{\i \l( \xi_1 \Xi_a + \xi_2 \Xi_b \r)} \r] \* \E_\nu \l[ e^{\i \l( \xi_1 Z_a + \xi_2 Z_b \r) } \r] d\xi_1 d\xi_2.
\end{align}
Now, since $\hat f_\ell$ is only supported on frequencies $|\xi| \gtrsim N$, the above integral will be small if the joint Fourier transform,
\begin{equation}
		\E_\nu \l[ e^{\i \l( \xi_1 Z_a + \xi_2 Z_b \r)} \r],
\end{equation}
decays fast enough. As $(Z_a, Z_b)$ are Gaussian, this amounts to controlling their variance and covariance.

The approximation (\ref{a:xa-decomp}) relies on the homogenization estimate Theorem \ref{t:LaSo22-2.6} and a clever pairwise coupling procedure from \cite{LaSo22}. In particular, fix a pair $(a, b) \in [\![N]\!]^2$. Let $\a, \d_1 > 0$ be fixed and we may focus on treating the contribution from pairs that are both in the bulk $I_\a^2$ and satisfying
\begin{equation}
		|a - b| > \frac{N^{\d_1}}{N\h}.
\end{equation}
Those pairs either not in the bulk or that are close together may be estimated identically as in \cite{LaSo22}*{Propositions 4.1 \& 4.2}. Let $(\xx(t))_{t \geq 0}$ and $(\yy(t))_{t \geq 0}$ be two DBM processes (\ref{e:dbm-sde}) each at equilbrium. However, instead of coupling them with the same Brownian motion for all times $t \geq 0$, we fix a time $T_{ab} > 0$ (to be specified later) and let $(\xx(t))_{0 \leq t \leq T_{ab}}$ and $(\yy(t))_{0 \leq t \leq T_{ab}}$ be driven by \emph{independent} Brownian motions $B^{\xx}$ and $B^{\yy}$,
\begin{equation}
		(B^{\xx}(t))_{0 \leq t \leq T_{ab}} \indep (B^{\yy}(t))_{0 \leq t \leq T_{ab}}.
\end{equation}
Then, for times $t > T_{ab}$ we again enforce that $(\xx(t))_{t \geq T_{ab}}$ and $(\yy(t))_{t \geq T_{ab}}$ are driven by the \emph{same} Brownian motion $B = (B_t)_{t \geq T_{ab}}$. We also specify that the initial conditions $\xx(0)$ and $\yy(0)$ start as independent. Note that the three families (among other combinations) of random variables,
\begin{equation}
		( \yy(T_{ab} + t))_{t > 0}, \qquad \xx(0), \qquad (B^{\xx}(t))_{0 \leq t \leq T_{ab}},
\end{equation}
are mutually independent. 

By Theorem \ref{t:LaSo22-2.6} we have for any $\e > 0$ that with overwhelming probability,
\begin{equation}
		\l|\l[ x_a(T_{ab} + t) - y_a(T_{ab} + t)\r]  - \l[ X(\ga_a, t) - Y(\ga_a, t) \r]  \r| \leq \frac{N^\e}{N^2 t},
\end{equation}
where we redefine $X$ and $Y$ to account for time ``starting'' at $T_{ab}$,
\begin{align}
		X(z, t) &:= \frac{1}{N \Im m_{\t{sc}}(z_t)} \l[ \Im \sum_{k=1}^N \log \l( x_k(T_{ab}) - z_t \r) - N \int \Im \log (x - z_t) \rho_{\t{sc}}(x) dx \r], \\
		Y(z, t) &:= \frac{1}{N \Im m_{\t{sc}}(z_t)} \l[ \Im \sum_{k=1}^N \log \l( y_k(T_{ab}) - z_t \r) - N \int \Im \log (x - z_t) \rho_{\t{sc}}(x) dx \r].
\end{align}
Now, we may further break down $X(\ga_a, t)$ via the method of characteristics on the range $[0, T_{ab}]$; by \cite{LaSo22}*{Proposition 7.2} it holds with overwhelming probability,
\begin{equation}
		X(\ga_a, t) = \tilde X_a(t) + Z_a(t) + O \l( \frac{N^\e}{N^2 t} \r),
\end{equation}
where
\begin{align}
		\tilde X_a(t) &:= \frac{1}{N \Im \l[ m_{\t{sc}} \l( (\ga_a)_{t} \r) \r]} \l[ \Im \sum^{N}_{k=1} \log \l( x_k(0) - (\ga_a)_{T_{ab} + t} \r) - N \int \Im \log (x - (\ga_a)_{T_{ab} + t}) d\rho_{\t{sc}}(x)\r] \\
					  &\quad+ \frac{1}{2} \l( 1 - \frac{2}{\b} \r) \int_0^{T_{ab}} \Im \l[ m_{\t{sc}}'\l((\ga_a)_{t + s}\r) \r] ds,
\end{align}
and
\begin{align}
		Z_a(t) := \frac{1}{N} \sqrt{ \frac{2}{N\b} } \sum^{N}_{k=1} \int_0^{T_{ab}} \Im \frac{1}{\ga_k - (\ga_a)_{T_{ab} + t - s}} dB^{\xx}_k(s).
\end{align}
Similar definitions and approximations can be made by swapping $a$ with $b$. Throughout we will hide the dependence on the pair $(a, b)$ in the notation and in the following proofs we make the assumption that all quantities as defined above depend on the pairwise coupling. The following is a restatement of \cite{LaSo22}*{Proposition 4.3} in our notation.

\begin{proposition}\label{p:x-Xi-Z}
		Let $\e > 0$. With overwhelming probability, we have (and the analogous estimate for $x_b(T_{ab} + t)$),
		\begin{equation}\label{e:x-Xi-Z}
				\l| x_a(T_{ab} + t) - \l( \Xi_a(t) + Z_a(t) \r) \r| \leq \frac{N^\e}{N^2 t},
		\end{equation}
		where
		\begin{equation}
				\Xi_a(t) := y_a(T_{ab} + t) - Y(\ga_a, t) + \tilde X_a(t).
		\end{equation}
		The families of random variables,
		\begin{equation}
				(Z_a(t), Z_b(t))_{t \geq 0}, \qquad (\Xi_a(t), \Xi_b(t))_{t \geq 0},
		\end{equation}
		are independent, and the $(Z_a, Z_b)$ are jointly Gaussian. Moreover, we have the bounds
		\begin{equation}
				\Var(Z_a), \Var(Z_b) \asymp (\d_2 - \d_3) \frac{\log N}{N^2},
		\end{equation}
		and
		\begin{equation}
				\Cov(Z_a, Z_b) \lesssim \frac{1}{N^2\log N}.
		\end{equation}
\end{proposition}

The following proposition justifies the approximation (\ref{a:xa-decomp-Cov}). The proof is identical to that of \cite{LaSo22}*{Proposition 4.4} but we include it here for completeness since technically speaking their result is stated for Gaussian divisible ensembles whereas we state ours for general (Gaussian) $\b$-ensembles.

\begin{lemma}[Stability of covariance for $N^{-2} \lesssim \h \lesssim N^{-1}$]\label{l:cov-stab-submicro}
		Let $0 < \d_3 < \d_2 < \d_1$ and we set
		\begin{equation}
				t_{ab} = N^{-\d_2} |\ga_a - \ga_b|, \quad T_{ab} = N^{-\d_3} |\ga_a - \ga_b|.
		\end{equation}
		Then for any $g \in L^2(\R)$ it holds
		\begin{align}
				&\sum_{(a, b) \in I_\a^2} \1_{|a - b| > \frac{N^{\d_1}}{N \h}}\l| \int_{\R^2} g(E_z) g(E_w) \Cov \l( \Im \frac{1}{x_a - z}, \Im \frac{1}{x_b - w} \r) dE_z dE_w \r. \\
				&\quad - \l.\int_{\R^2} \Cov \l( \Im \frac{1}{\Xi_a + Z_a - z}, \Im \frac{1}{\Xi_b + Z_b - w} \r) dE_z dE_w \r| \lesssim \frac{N^{3\e + \d_2}}{\h} \lV g \rV_2^2
		\end{align}
		where we denoted $z = E_z + \i\h$ and $w = E_w + \i\h$ above.
\end{lemma}

\begin{proof}
		For convenience, denote
		\begin{equation}
				\tilde \Xi_a(t) := \Xi_a(t) + Z_a(t),
		\end{equation}
		and likewise for the index $b$. We remind the reader that $\tilde \Xi_a = \tilde \Xi_a^{(ab)}$ (and likewise $\Xi_a$, $Z_a$, etc.) depends on the pair $(a,b)$, and when summing over $(a,b) \in I_\a^2$ we use a separate coupling for each pair. From (\ref{e:x-Xi-Z}) and the restriction $|a - b| > \frac{N^{\d_1}}{N\h}$ we have
		\begin{equation}
				|x_a(T_{ab} + t) - \tilde \Xi_a(t)| \leq \frac{N^\e}{N^2 t_{ab}} \ll \h,
		\end{equation}
		which implies
		\begin{equation}
				\frac{1}{|x - z|^2} \lesssim \frac{1}{|\tilde \Xi_a(t) - z|^2}
		\end{equation}
		for any $x$ between $x_a(T_{ab} + t)$ and $\tilde \Xi_a(t)$. Therefore we may expand the covariance as
		\begin{align}
				&\l| \Cov \l( \Im \frac{1}{x_a - z}, \Im \frac{1}{x_b - w} \r) - \Cov \l( \Im \frac{1}{\tilde \Xi_a - z}, \Im \frac{1}{\tilde \Xi_b - w} \r) \r| \nonumber \\
				&\quad \leq \l| \Cov \l( \Im \frac{1}{x_a - z}, \Im \frac{1}{x_b - w} \r) - \Cov \l( \Im \frac{1}{\tilde \Xi_a - z}, \Im \frac{1}{x_b - w} \r) \r| \nonumber\\
				&\quad\quad+ \l| \Cov \l( \Im \frac{1}{\tilde \Xi_a - z}, \Im \frac{1}{x_b - w} \r) - \Cov \l( \Im \frac{1}{\tilde \Xi_a - z}, \Im \frac{1}{\tilde \Xi_b - w} \r) \r| \nonumber \\
				&\quad \lesssim \frac{N^\e}{N^2 t_{ab}}\E \l[ \l( \l.\frac{d}{dx} \r|_{x \in (x_a, \tilde \Xi_a)} \Im \frac{1}{x - z} \r)  \* \l( \Im  \frac{1}{x_b - w} - \E \l[ \Im \frac{1}{x_b - w} \r] \r) \r] \nonumber\\
				&\quad\quad+ \frac{N^\e}{N^2 t_{ab}} \E \l[ \l| \frac{1}{\tilde \Xi_a - z} - \E \l[ \frac{1}{\tilde \Xi_a - z} \r] \r| \* \frac{1}{|\tilde \Xi_b - w|^2} \r] \nonumber\\
				&\quad \lesssim \frac{N^\e}{N^2 t_{ab}} \E \l[ \l( \l.\frac{d}{dx} \r|_{x \in (x_a, \tilde \Xi_a)} \Im \frac{1}{x - z} \r)  \* \l( \Im  \frac{1}{\tilde \Xi_b - w} - \E \l[ \Im \frac{1}{\tilde \Xi_b - w} \r] \r) \r] \nonumber\\
				&\quad\quad+ \frac{N^{2\e}}{N^4 t_{ab}^2}\E \l[ \l( \l.\frac{d}{dx} \r|_{x \in (x_a, \tilde \Xi_a)} \Im \frac{1}{x - z} \r)  \* \l( \l. \frac{d}{dx} \r|_{x \in (x_b, \tilde \Xi_b)} \Im  \frac{1}{x - w} - \E \l[ \l. \frac{d}{dx} \r|_{x \in (x_b, \tilde \Xi_b)}\Im \frac{1}{x_b - w} \r] \r) \r] \nonumber \\
				&\quad\quad+ \frac{N^\e}{N^2 t_{ab}} \E \l[ \l| \frac{1}{\tilde \Xi_a - z} - \E \l[ \frac{1}{\tilde \Xi_a - z} \r] \r| \* \frac{1}{|\tilde \Xi_b - w|^2} \r] \nonumber\\ 
				&\quad \lesssim \frac{N^\e}{N^2 t_{ab}} \l\{ \E \l[ \frac{1}{|\tilde \Xi_a - z|^2} \* \l| \frac{1}{ \tilde \Xi_b - w} - \E \l[ \frac{1}{\tilde \Xi_b - w} \r] \r|   \r] + \E \l[ \l| \frac{1}{\tilde \Xi_a - z} - \E \l[ \frac{1}{\tilde \Xi_a - z} \r] \r| \* \frac{1}{|\tilde \Xi_b - w|^2} \r] \r\} \label{e:dCov-1} \\
				&\quad\quad+ \frac{N^{2\e}}{N^4 t^2_{ab}} \l\{\E \l[ \frac{1}{|\tilde \Xi_a - z|^2} \* \frac{1}{|\tilde \Xi_b - w|^2} \r] + \E \l[ \frac{1}{|\tilde \Xi_a - z|^2}  \r] \* \E \l[ \frac{1}{|\tilde \Xi_b - w|^2} \r]   \r\} \label{e:dCov-2}.
		\end{align}
		Due to the symmetry between the two terms in (\ref{e:dCov-1}), we may simply focus on the first term. Fixing a small $\e_1 > 0$, we may break the integral over $E_z$ and $E_w$ into four terms,
		\begin{equation}
				1 \equiv \l( \1_{|E_z - \ga_a| < \frac{N^{\e_1}}{N}} + \1_{|E_z - \ga_a| \geq \frac{N^{\e_1}}{N}} \r) \l( \1_{|E_w - \ga_b| < \frac{N^{\e_1}}{N}} +  \1_{|E_w - \ga_b| \geq \frac{N^{\e_1}}{N}} \r).
		\end{equation}
		The pertinent case is when both $E_z$ and $E_w$ are close to the corresponding typical locations. In this case we use the fact that
		\begin{equation}
				t_{ab} = N^{-\d_2} |\ga_a - \ga_b| \gtrsim \frac{N^{\d_1 - \d_2}}{N^2 \h} \gtrsim \frac{N^{\d_1 - \d_2}}{N},
		\end{equation}
		and hence so long as $\e_1 < \d_1 - \d_2$ it holds that
		\begin{equation}\label{e:tab-zw-close}
				t_{ab} \gtrsim N^{-\d_2} |E_z - E_w| + N^{-1}.
		\end{equation}
		Now, applying Young's inequality and Lemma \ref{l:log-div-1/Z} (note that we do not use any cancellation $(\tilde \Xi_b - w)^{-1}$ with its mean) we have
		\begin{align}
				&\sum_{\substack{(a,b) \in I_\a^2 \\ |a - b| > \frac{N^{\d_1}}{N\h}}}  \int_{\R^2} |g(E_z) g(E_w)| \1_{|E_z - \ga_a|, |E_w - \ga_b| < \frac{N^{\e_1}}{N}} \frac{N^\e}{N^2 t_{ab}}\E \l[ \frac{1}{|\tilde \Xi_a - z|^2} \* \l| \frac{1}{ \tilde \Xi_b - w} - \E \l[ \frac{1}{\tilde \Xi_b - w} \r] \r|   \r] dE_z dE_w  \\ 
				&\quad \lesssim \frac{N^\e}{N^2\h} \sum_{\substack{(a,b) \in I_\a^2 \\ |a - b| > \frac{N^{\d_1}}{N\h}}}  \int_{\R^2} |g(E_z)|^2 \1_{|E_z - \ga_a|, |E_w - \ga_b| < \frac{N^{\e_1}}{N}} \frac{1}{t_{ab}} \E \l[ \frac{1}{|\tilde \Xi_a - z|} \* \l| \frac{1}{ \tilde \Xi_b - w} - \E \l[ \frac{1}{\tilde \Xi_b - w} \r] \r|   \r] dE_z dE_w \\
				&\quad \lesssim \frac{N^{\e}}{N^2\h} \sum_{\substack{(a,b) \in I_\a^2 \\ |a - b| > \frac{N^{\d_1}}{N\h}}}  \int_{\R^2} |g(E_z)|^2 \1_{|E_z - \ga_a|, |E_w - \ga_b| < \frac{N^{\e_1}}{N}} \frac{N^{2+\e}}{t_{ab}} dE_z dE_w. \label{e:zw-close-1}
		\end{align}
		Applying (\ref{e:tab-zw-close}) we can further estimate
		\begin{align*}	
				(\ref{e:zw-close-1}) &\lesssim \frac{N^{2\e+\d_2}}{\h}  \sum_{\substack{(a,b) \in I_\a^2 \\ |a - b| > \frac{N^{\d_1}}{N\h}}} \int_{\R^2} |g(E_z)|^2 \1_{|E_z - \ga_a|, |E_w - \ga_b| < \frac{N^{\e_1}}{N}}\frac{1}{|E_z - E_w| + N^{-1}}  dE_z dE_w \\
				&\lesssim \frac{N^{2\e+\d_2 + 2\e_1}}{\h}  \int_{\R^2} |g(E_z)|^2 \frac{1}{|E_z - E_w| + N^{-1}} dE_z dE_w \\
				&\lesssim \frac{N^{2\e+\d_2 + 2\e_1}\log N}{\h} \lV g \rV_2^2
		\end{align*}
		where in the last step we integrated in $E_w$.

		Next, we consider the contribution when $|E_z - \ga_a| \geq N^{-1 + \e_1}$ and $|E_w - \ga_b| < N^{-1 + \e_1}$. Note that by rigidity it holds
		\begin{equation}\label{e:za-far}
				\frac{1}{|\tilde \Xi_a - (E_z + \i\h)|^2} \lesssim \frac{1}{|\tilde \Xi_a - (E_z + \i N^{-1})|^2}.
		\end{equation}
		Now, one of the contributions from the Young's inequality is given by
		\begin{align}	
				&\sum_{\substack{(a,b) \in I_\a^2 \\ |a - b| > \frac{N^{\d_1}}{N\h}}}  \int_{\R^2} |g(E_w)|^2 \1_{|E_z - \ga_a| \geq \frac{N^{\e_1}}{N}, |E_w - \ga_b| < \frac{N^{\e_1}}{N}} \frac{N^\e}{N^2 t_{ab}}\E \l[ \frac{1}{|\tilde \Xi_a - z|^2} \* \l| \frac{1}{ \tilde \Xi_b - w} - \E \l[ \frac{1}{\tilde \Xi_b - w} \r] \r|   \r] dE_z dE_w . \label{e:za-far-wb-close-1}
		\end{align}
		Applying (\ref{e:za-far}) and integrating in $E_z$ gives
		\begin{align}
				(\ref{e:za-far-wb-close-1}) &\lesssim \sum_{\substack{(a,b) \in I_\a^2 \\ |a - b| > \frac{N^{\d_1}}{N\h}}}  \int_{\R} |g(E_w)|^2 \1_{|E_w - \ga_b| < \frac{N^{\e_1}}{N}} \frac{N^\e}{N t_{ab}}\E \l[ \l| \frac{1}{ \tilde \Xi_b - w} - \E \l[ \frac{1}{\tilde \Xi_b - w} \r] \r|   \r] dE_w . \label{e:za-far-wb-close-2}
		\end{align}
		Next, using (\ref{e:1/Z+h}) to simplify the expectation,
		\begin{align}
				(\ref{e:za-far-wb-close-1}) &\lesssim N^{2\e} \sum_{\substack{(a,b) \in I_\a^2 \\ |a - b| > \frac{N^{\d_1}}{N\h}}}  \int_{\R} |g(E_w)|^2 \1_{|E_w - \ga_b| < \frac{N^{\e_1}}{N}} \frac{1}{t_{ab}} dE_w. \label{e:za-far-wb-close-3} \\
											&\lesssim N^{1 + 2\e + \d_2} \log N \sum_{b \in I_\a} \int_\R |g(E_w)|^2 \1_{|E_w - \ga_b| < \frac{N^{\e_1}}{N}} dE_w \\
											&\lesssim N^{1 + 2\e + \d_2 + \e_1} \lV g \rV_2^2,
		\end{align}
		where in the second inequality we used that for all $b \in I_\a$,
		\begin{equation}\label{e:1/tab-summable}
				\sum_{a \in I_\a} \1_{|a - b| > \frac{N^{\d_1}}{N\h}} \frac{1}{t_{ab}} \lesssim N^{1 + \d_2} \log N,
		\end{equation}
		and in the third inequality we used that for all $E_w$,
		\begin{equation}
				\sum_{b \in I_\a} \1_{|E_w - \ga_b| < \frac{N^{\e_1}}{N}} \lesssim N^{\e_1}.
		\end{equation}
		On the other hand, the other contribution from the Young's inequality is bounded by
		\begin{align}	
				&\sum_{\substack{(a,b) \in I_\a^2 \\ |a - b| > \frac{N^{\d_1}}{N\h}}}  \int_{\R^2} |g(E_z)|^2 \1_{|E_z - \ga_a| \geq \frac{N^{\e_1}}{N}, |E_w - \ga_b| < \frac{N^{\e_1}}{N}} \frac{N^\e}{N^2 t_{ab}}\E \l[ \frac{1}{|\tilde \Xi_a - z|^2} \* \l| \frac{1}{ \tilde \Xi_b - w} - \E \l[ \frac{1}{\tilde \Xi_b - w} \r] \r|   \r] dE_z dE_w .  \\
				&\quad \lesssim \sum_{\substack{(a,b) \in I_\a^2 \\ |a - b| > \frac{N^{\d_1}}{N\h}}}  \int_{\R^2} |g(E_z)|^2  \1_{|E_z - \ga_a| \geq \frac{N^{\e_1}}{N}} \frac{N^\e \log N}{N^2 t_{ab}} \E \l[ \frac{1}{|\tilde \Xi_a - z|^2} \r] dE_z \\
				&\quad \lesssim \sum_{\substack{(a,b) \in I_\a^2 \\ |a - b| > \frac{N^{\d_1}}{N\h}}}  \int_{\R} |g(E_z)|^2 \1_{|E_z - \ga_a| \geq \frac{N^{\e_1}}{N}} \frac{N^\e \log N}{N^2 t_{ab}} \* \frac{1}{(\ga_a - E_z)^2 + N^{-2}} dE_z \\
				&\quad \lesssim \frac{N^{\e + \d_2} (\log N)^2}{N}\sum_{a \in I_\a} \int_{\R} |g(E_z)|^2 \1_{|E_z - \ga_a| \geq \frac{N^{\e_1}}{N}} \frac{1}{(\ga_a - E_z)^2 + N^{-2}} dE_z \\
				&\quad \lesssim N^{1 + \e + \d_2} (\log N)^2 \lV g \rV_2^2,
		\end{align}
		where we integrated in $E_w$ to get the first inequality, applied (\ref{e:za-far}) together with rigidity in the second, used (\ref{e:1/tab-summable}) but for the sum over $b$ in the third, and the bound
		\begin{equation}\label{e:1/N-cauchy-kernel}
				\frac{1}{N} \sum_{a \in I_\a} \frac{1}{(\ga_a - E_z)^2 + N^{-2}} \lesssim N
		\end{equation}
		to get the last inequality. 

		We now cover the case where $|E_z - \ga_a| < N^{-1 + \e_1}$ but $|E_w - \ga_b| \geq N^{-1 + \e_1}$. For such $E_w$, it holds by rigidity that
		\begin{equation}
				\l| \frac{1}{\tilde \Xi_b - w} - \E \l[ \frac{1}{\tilde \Xi_b - w} \r] \r| \lesssim \frac{N^\e}{N} \frac{1}{(\ga_b - E_w)^2 + N^{-2}}.
		\end{equation}
		Then, applying (\ref{e:1/Z+h-2}) we are led to estimate
		\begin{align}
				&\sum_{\substack{(a,b) \in I_\a^2 \\ |a - b| > \frac{N^{\d_1}}{N\h}}}  \int_{\R^2} |g(E_z)g(E_w)| \1_{|E_z - \ga_a| < \frac{N^{\e_1}}{N}, |E_w - \ga_b| \geq \frac{N^{\e_1}}{N}} \frac{N^{\e}}{N^3 t_{ab}} \frac{1}{(\ga_b - E_w)^2 + N^{-2}} \E \l[ \frac{1}{|\tilde \Xi_a - z|^2} \r] dE_z dE_w \\
				&\quad \lesssim  \sum_{\substack{(a,b) \in I_\a^2 \\ |a - b| > \frac{N^{\d_1}}{N\h}}}  \int_{\R^2} |g(E_z)g(E_w)| \1_{|E_z - \ga_a| < \frac{N^{\e_1}}{N}, |E_w - \ga_b| \geq \frac{N^{\e_1}}{N}} \frac{N^{2\e}}{N^2 \h t_{ab}} \frac{1}{(\ga_b - E_w)^2 + N^{-2}} dE_z dE_w. \label{e:za-close-wb-far-1}
		 \end{align}
		 As in the previous case, we apply Young's inequality and bound the two terms separately. For the $|g(E_w)|^2$ term, we write
		 \begin{align}
				&\sum_{\substack{(a,b) \in I_\a^2 \\ |a - b| > \frac{N^{\d_1}}{N\h}}}  \int_{\R^2} |g(E_w)|^2 \1_{|E_z - \ga_a| < \frac{N^{\e_1}}{N}, |E_w - \ga_b| \geq \frac{N^{\e_1}}{N}} \frac{N^{2\e}}{N^2 \h t_{ab}} \frac{1}{(\ga_b - E_w)^2 + N^{-2}} dE_z dE_w \\
				&\quad\lesssim \frac{N^{2\e + \e_1}}{N^2 \h} \sum_{\substack{(a,b) \in I_\a^2 \\ |a - b| > \frac{N^{\d_1}}{N\h}}}  \int_{\R} |g(E_w)|^2 \1_{|E_w - \ga_b| \geq \frac{N^{\e_1}}{N}} \frac{1}{t_{ab}} \frac{1}{(\ga_b - E_w)^2 + N^{-2}} dE_w \\
				&\quad\lesssim  \frac{N^{2\e + \e_1 + \d_1}\log N}{N \h} \sum_{b \in I_\a} \int_\R |g(E_w)|^2 \1_{|E_w - \ga_b| \geq \frac{N^{\e_1}}{N}} \frac{1}{(\ga_b - E_w)^2 + N^{-2}} dE_w  \\
				&\quad \lesssim \frac{N^{2\e + \e_1 + \d_1}\log N}{\h} \lV g \rV_2^2,
		 \end{align}
		 where we integrated in $E_z$ to get the first inequality, applied (\ref{e:1/tab-summable}) in the second, and integrated in $E_w$ to get the third inequality. For the other contribution from $|g(E_z)|^2$, write
		 \begin{align} 
				&\sum_{\substack{(a,b) \in I_\a^2 \\ |a - b| > \frac{N^{\d_1}}{N\h}}}  \int_{\R^2} |g(E_z)|^2 \1_{|E_z - \ga_a| < \frac{N^{\e_1}}{N}, |E_w - \ga_b| \geq \frac{N^{\e_1}}{N}} \frac{N^{2\e}}{N^2 \h t_{ab}} \frac{1}{(\ga_b - E_w)^2 + N^{-2}} dE_z dE_w \\
				&\quad\lesssim \frac{N^{2\e}}{N\h}\sum_{\substack{(a,b) \in I_\a^2 \\ |a - b| > \frac{N^{\d_1}}{N\h}}}  \int_{\R} |g(E_z)|^2 \1_{|E_z - \ga_a| < \frac{N^{\e_1}}{N}, |E_w - \ga_b| \geq \frac{N^{\e_1}}{N}} \frac{1}{t_{ab}}  dE_z \\ 
				&\quad\lesssim \frac{N^{2\e + \d_2}\log N}{\h} \sum_{a \in I_\a} \int_\R |g(E_z)|^2 \1_{|E_z - \ga_a| < \frac{N^{\e_1}}{N}} dE_z \\
				&\quad\lesssim \frac{N^{2\e + \d_2 +\e_1}\log N}{\h} \lV g \rV_2^2,
		 \end{align}
		 where we integrated in $E_w$ to get the first inequality, applied (\ref{e:1/tab-summable}) in the second, and summed over $a$ to get the third inequality.

		As the last case for the term (\ref{e:dCov-1}), we consider when both $|E_z - \ga_a| \geq N^{-1+ \e_1}$ and $|E_2 - \ga_b| \geq N^{-1 + \e_1}$. When this happens we have by rigidity that
		\begin{align}
				\1_{|E_z - \ga_a|, |E_2 - \ga_b| \geq \frac{N^{\e_1}}{N}}  \E \l[ \frac{1}{|\tilde \Xi_a - z|^2} \* \l| \frac{1}{ \tilde \Xi_b - w} - \E \l[ \frac{1}{\tilde \Xi_b - w} \r] \r|   \r] &\lesssim \frac{N^\e}{N} \frac{1}{(\ga_a - E_z)^2 + N^{-2}} \frac{1}{(\ga_b - E_w)^2 + N^{-2}}.
		\end{align}
		Then, we may bound
		\begin{align}
				&\sum_{\substack{(a,b) \in I_\a^2 \\ |a - b| > \frac{N^{\d_1}}{N\h}}} \int_{\R^2} |g(E_z) g(E_w)| \frac{N^{2\e}}{N^3 t_{ab}} \frac{1}{(\ga_a - E_z)^2 + N^{-2}} \frac{1}{(\ga_b - E_w)^2 + N^{-2}}  dE_z dE_w \\ 
				&\quad\lesssim \sum_{\substack{(a,b) \in I_\a^2 \\ |a - b| > \frac{N^{\d_1}}{N\h}}} \int_{\R^2} |g(E_z)|^2 \frac{N^{2\e}}{N^3 t_{ab}} \frac{1}{(\ga_a - E_z)^2 + N^{-2}} \frac{1}{(\ga_b - E_w)^2 + N^{-2}}  dE_z dE_w \\
				&\quad\lesssim \frac{N^{2\e}}{N^2} \sum_{\substack{(a,b) \in I_\a^2 \\ |a - b| > \frac{N^{\d_1}}{N\h}}} \int_{\R} |g(E_z)|^2 \frac{1}{t_{ab}} \frac{1}{(\ga_a - E_z)^2 + N^{-2}}  dE_z \\
				&\quad \lesssim  \frac{N^{2\e + \d_2}\log N}{N} \sum_{a \in I_\a} \int_\R |g(E_z)|^2 \frac{1}{(\ga_a - E_z)^2 + N^{-2}}  dE_z \\
				&\quad \lesssim N^{1 + 2\e + \d_2} \log N \lV g \rV_2^2,
		\end{align}
		where we integrated in $E_w$ to get the first inequality, applied (\ref{e:1/tab-summable}) for the sum in $b$ in the second, and (\ref{e:1/N-cauchy-kernel}) to get the third.

		We now estimate the contribution from (\ref{e:dCov-2}), by writing
		\begin{align*}
				& \sum_{(a, b) \in I_\a^2} \1_{|a - b| > \frac{N^{\d_1}}{N\h}} \frac{N^{2\e}}{N^4 t_{ab}^2}\int_{\R^2} \l| g(E_z) g(E_w) \r| \E \l[ \frac{1}{|\tilde \Xi_a  - z|^2} \* \frac{1}{|\tilde \Xi_b - w|^2} \r] dE_z dE_w \\
				&\quad \lesssim \sum_{(a, b) \in I_\a^2} \1_{|a - b| > \frac{N^{\d_1}}{N\h}} \frac{N^{2\e}}{N^4 t_{ab}^2}\int_{\R^2} |g(E_z)|^2 \E \l[ \frac{1}{|\tilde \Xi_a  - z|^2} \* \frac{1}{|\tilde \Xi_b - w|^2} \r] dE_z dE_w \\
				&\quad \lesssim \sum_{(a, b) \in I_\a^2} \1_{|a - b| > \frac{N^{\d_1}}{N\h}} \frac{N^{2\e}}{N^4 t_{ab}^2\h}\int_{\R} |g(E_z)|^2 \E \l[ \frac{1}{|\tilde \Xi_a  - z|^2} \r] dE_z \\ 
				&\quad \lesssim \frac{N^{2\e + 2\d_2 - \d_1}}{N} \sum_{a \in I_\a} \int_{\R} |g(E_z)|^2  \E \l[ \frac{1}{|\tilde \Xi_a  - z|^2} \r] dE_z \\
				&\quad \lesssim \frac{N^{2\e + 2\d_2 - \d_1}}{N\h}\sum_{a \in I_\a} \int_{\R} |g(E_z)|^2  \E \l[ \frac{1}{|\tilde \Xi_a  - z|} \r] dE_z \\ 
				&\quad \lesssim \frac{N^{2\e + 2\d_2 - \d_1}}{N\h} \int_\R |g(E_z)|^2 \l\{ \sum_{\substack{a \in I_\a \\ |\ga_a - E_z| < \frac{N^\e}{N}}} \E \l[ \frac{1}{|\tilde \Xi_a - z|} \r] + \sum_{\substack{a \in I_\a \\ |\ga_a - E_z| \geq \frac{N^\e}{N}}} \frac{1}{|\ga_a - E_z|} \r\} \\
				&\quad \lesssim \frac{N^{2\e + 2\d_2 - \d_1}}{N\h} \int_\R |g(E_z)|^2 \l\{ \frac{N^{1+\e}}{\log N} \log \l( \frac{\log N}{N\h } \r) + N \log N \r\} \\
				&\quad \lesssim \frac{N^{3\e + 2\d_2 - \d_1}\log N}{\h} \lV g \rV_{2}^2,
		\end{align*}
		where we applied Young's inequality in the first inequality, integrated over $E_w$ in the second, used
		\begin{equation}
				\sum_{b \in I_\a} \1_{|a - b| > \frac{N^{\d_1}}{N\h}} \frac{1}{N^4 t_{ab}^2} \lesssim \frac{N^{2\d_2}}{N^2}\sum_{b \in I_\a} \1_{|a - b| > \frac{N^{\d_1}}{N\h}} \frac{1}{|a - b|^2} \lesssim \frac{N^{2\d_2 - \d_1}  \h}{N}
		\end{equation}
		in the third, rigidty in the fourth, and (\ref{e:1/Z+h}) in the second to last inequality. The estimate on the second term in (\ref{e:dCov-2}) is similar and in fact the proof is identical after the second inequality.
\end{proof}

\begin{lemma}\label{l:cov-Xi-tiny}
		Let $\d_1 > \d_2 > \d_3 > 0$ be as above. Then for some $B_1 = B_1(\d_1, \d_2, \d_3) > 0$ large enough, and
		\begin{equation}
				\frac{1}{N^{10}} \leq \h \leq \frac{1}{B_1 N},
		\end{equation}
		it holds
		\begin{equation}
				\l| \int_{\R^2} g_\ell(E_z) g_\ell(E_w) \Cov \l( \frac{1}{\tilde \Xi_a - z}, \Im \frac{1}{\tilde \Xi_b - w} \r) dE_z dE_w \r| \lesssim N^{-100} \lV g_\ell\rV_2^2.
		\end{equation}
\end{lemma}

\subsection{Very high frequencies}%
\label{sub:regime3}

In this section we recall a Wegner estimate and use it to cover pairs $(a,b)$ that are closer than the scale $ \frac{N^{\d_1}}{N\h}$, for any $\h \lesssim N^{-1}$. Not only does this cover the remaining pairs not counted in Lemma \ref{l:cov-stab-submicro}, but it also covers the case of very high frequencies $\h \lesssim N^{-2}$, for which $(N\h)^{-1}$ is actually larger than the total number of particles (in which case the statement of Lemma \ref{l:cov-stab-submicro} is vacuous).

\begin{lemma}[Wegner estimate]\label{l:wegner-beta}
		Let $(x_k)_{k=1}^N$ be distributed as a Gaussian $\b$-ensemble for $1 \leq \b \leq 100$. Let $\kappa > 0$ and $0 < \h < 1/N$. There is a constant $N_0 > 0$ such that for all  $N \geq N_0$ it holds
		\begin{equation}
				\P \l( \exists k: |x_k - E| < \h \r) \lesssim N \h,
		\end{equation}
		for all $E \in (-2 + \kappa, 2 - \kappa)$.
\end{lemma}

\begin{proof}
		This follows immediately from a Markov inequality,
		\begin{equation}
				\P \l( \exists k : |x_k - E| < \h \r) = \P \l( \Nor([E - \h, E + \h]) \geq 1 \r) \leq \E \l[ \Nor([E - \h, E + \h]) \r] \lesssim N \h,
		\end{equation}
		where in the last inequality we applied Lemma \ref{l:wegner} from the appendix.
\end{proof}

The following propositions now follow from the same proofs of \cite{LaSo22}*{Propositions 2.13 \& 4.1, 4.2} using \cite{LaSo22}*{Theorem 2.11} (which is stated above as our Lemma \ref{l:wegner-beta}) by just replacing the words ``smooth Wigner matrix'' with ``Gaussian $\b$-ensemble''. For sake of completeness, we briefly summarize their proofs from \cite{LaSo22}. Recall that we used these propositions to handle the \emph{very} small scales $N^{-A} \lesssim \h \lesssim N^{-2 + \d_1}$ in our proof of Theorem \ref{t:var-H1/2}.

\begin{proposition}
		Let $(x_k)_{k=1}^N$ be distributed as a Gaussian $\b$-ensemble for $1 \leq \b \leq 100$. Let $A, \a, \e > 0$. Then it holds
		\begin{equation}\label{e:sup-bulk-E-est}
				\sup_{E \in \R} \sum_{\a N \leq k \leq (1 - \a) N} \E \l[ \frac{1}{N |x_k - (E + \i\h)|} \r]\lesssim N^\e,
		\end{equation}
		for all $\h \geq N^{-A}$.
\end{proposition}

\begin{proof}
		In \cite{LaSo22}, they first prove the estimate
		\begin{equation}\label{e:one-point-E-est}
				\E \l[ \frac{1}{N |x_k - (E + \i\h)|} \r] \lesssim N^\e,
		\end{equation}
		for all $E \in \R$ and $\h \geq N^{-A}$ and $k \in [\![ \a N, (1 - \a) N]\!]$. For the $CN^\e$ indices close to $E$, satisfying $|\ga_k - E| \leq N^{\e - 1}$, they apply (\ref{e:one-point-E-est}), whereas for the remaining indices a rigidity estimate suffices,
		\begin{equation}
				\sum_{\a N \leq k \leq (1 - \a) N} \1_{|\ga_k - E| > N^{\e -1}} \E \l[ \frac{1}{N|x_k - (E + \i\h)|} \r] \lesssim \sum^{N}_{k=1} \frac{1}{k} \lesssim \log N.
		\end{equation}
		Together, the claim follows.
\end{proof}

\begin{proposition}\label{p:submicro-close}
		Let $(x_k)_{k=1}^N$ be distributed as a Gaussian $\b$-ensemble, for any (fixed) $\b > 0$. Let $A > 1$ (large), $\d_1 , \e > 0$ (small) and
		\begin{equation}
				N^{-A} \leq \h \leq N^{-1}.
		\end{equation}
		Then for any function $g \in L^2(\R)$ it holds
		\begin{align}
				\sum_{\substack{(a,b) \in I_\a^2 \\ |a - b| < \frac{N^{\d_1}}{N\h}}} \int_{\R^2} |g(E_z) g(E_w)| \Cov \l( \Im \frac{1}{x_a - (E_z + \i \h)}, \Im \frac{1}{x_b - (E_w + \i\h)} \r) dE_z dE_w \lesssim \frac{N^{\d_1 + \e}}{\h} \lV g \rV_2^2.
		\end{align}
\end{proposition}

\begin{proof}
		In \cite{LaSo22}, they bound the two terms of $\Cov(X, Y) = \E[XY] - \E[X] \E[Y]$ separately, both in the same way. For example, the second term may be estimated by writing, using Young's inequality and integrating out the Cauchy kernel,
		\begin{align*}
				&\sum_{\substack{(a,b) \in I_\a^2 \\ |a - b| < \frac{N^{\d_1}}{N\h}}} \int_{\R^2} |g(E_z) g(E_w) | \E \l[ \Im \frac{1}{x_a - (E_z + \i\h)} \r]\E \l[ \Im \frac{1}{x_b - (E_w + \i\h)} \r] dE_z dE_w \\
				&\lesssim \sum_{\substack{(a,b) \in I_\a^2 \\ |a - b| < \frac{N^{\d_1}}{N\h}}} \int_{\R^2} |g(E_z)|^2 \E \l[ \Im \frac{1}{x_a - (E_z + \i\h)} \r] dE_z \\
				&\lesssim \frac{N^{\d_1}}{N\h} \sum_{a \in I_\a} \int_\R |g(E_z)|^2 \E \l[ \Im \frac{1}{x_a - (E_z + \i\h)} \r] dE_z.
		\end{align*}
		Finally, an application of (\ref{e:sup-bulk-E-est}) yields the desired estimate.
\end{proof}

\begin{proposition}\label{p:submicro-edge}
		Let $\e, \kappa > 0$ (small) and $M, A > 0$ (large). Suppose $f$ is supported in $(-2 + \kappa, 2 + \kappa)$ and that $\h = 2^{-\ell}$ satisfies
		\begin{equation}
				N^{-A} \leq \h \leq N^{-1}.
		\end{equation}
		Also assume that $\a > 0$ is small enough so that $\ga_{\a N} < -2 + \kappa /4$. Then it holds
		\begin{multline}
				\l| \sum_{(a,b) \not\in I_\a^2} \int_{\R^2} g_\ell(E_z) g_\ell(E_w) \Cov \l( \Im \frac{1}{x_a - (E_z + \i\h)}, \Im \frac{1}{x_b - (E_w + \i\h)} \r) dE_z dE_w \r| \\
				\lesssim N^\e \l( \frac{1}{\h} \lV g_\ell \rV_2^2 + 2^{-M \ell} N^2 \lV f \rV_2^2 \r).
		\end{multline}
\end{proposition}

\begin{proof}
		Note that Proposition \ref{p:LP-decay} enables us to restrict the integration to the domain $dE_z dE_w \in (-2 + \kappa / 2, 2 - \kappa/2 )^2$. Together with rigidity and standard local law estimates (e.g. \cite{BoMoPa22}*{Theorem 1.1}), this separation of bul and edge is enough to handle this edge contribution.
\end{proof}

\appendix

\section{Wegner estimate}%
\label{app:wegner}

\begin{lemma}\label{l:wegner}
		Fix a $\kappa > 0$. There exists a $C > 0$ such that for any $I \subset (-2 + \kappa, 2 - \kappa)$ with  $|I| \lesssim N^{-1}$, it holds
		\begin{equation}
				\E \l[ \Nor(I) \r] \leq C N |I|.
		\end{equation}
		In particular, the constant does not depend on the choice of interval $I$ in the bulk.
\end{lemma}

\begin{proof}
		Note that the microscopic case is covered by \cite{BoMoPa22}*{Lemma 3.3}, i.e. it holds
		\begin{equation}\label{e:wegner-micro}
				\E[\Nor(I)] \leq O(1) \quad \t{ for } \quad |I| \asymp N^{-1}, \quad I \subset (-2 + \kappa, 2 - \kappa).
		\end{equation}
		So, it remains to prove for submicroscopic intervals $|I| \ll N^{-1}$. Recall that we let
		\begin{equation}
				\rho_N(\xx) := \frac{1}{Z_N} \exp \l( - \frac{N\b}{4} |\xx|^2 \r) \prod_{i < j} |x_i - x_j|^\b 
		\end{equation}
		denote the density of the particles. Note that it is enough to show that the density is not too sensitive to translations,
		\begin{equation}\label{e:rho-trans-invar}
				\rho_N(\xx) \lesssim N \int_{-1/N}^{1/N} \rho_N(\xx - t \bm 1) dt.
		\end{equation}
		Indeed, if (\ref{e:rho-trans-invar}) holds and $|I| \leq 1/N$ then we may write
		\begin{align*}
				\E \l[ \Nor(I) \r] &= \int_{\R^N} \Nor(I)(\xx) \rho_N(\xx) d\xx \\
								   &\lesssim N \int_{-1/N}^{1/N} \int_{\R^N} \Nor(I)(\xx) \rho_N(\xx - t \bm 1) d\xx dt \\
								   &= N \int_{-1/N}^{1/N} \int_{\R^N} \Nor(I)(\xx + t\bm 1) \rho_N(\xx) d\xx dt \\ 
								   &= N \int_{\R^N} \rho_N(\xx) \sum_{i=1}^N \int_{-1/N}^{1/N} \1_I(x_i + t) dt d\xx \\
								   &\leq N|I| \int_{\R^N} \rho_N(\xx) \Nor(I + [-1/N, 1/N]) d\xx \\
								   &\lesssim N |I|
		\end{align*}
		where in the last inequality we applied the microscopic Wegner estimate (\ref{e:wegner-micro}) from \cite{BoMoPa22}*{Lemma 3.3}. 

		Let us now prove (\ref{e:rho-trans-invar}). The Vandermonde term is translation invariant, so we have
		\begin{align*}
				\rho_N(\xx - t \bm 1) &= \frac{1}{Z_N} \exp \l( - \frac{N\b}{4}|\xx - t \bm 1|^2 \r) \prod_{i < j} |x_i - x_j|^\b \\
									  &= \rho_N(\xx) \exp \l( t\frac{N\b}{2}\<\xx, \bm 1\> - t^2 \frac{N^2\b}{4} \r).
		\end{align*}
		In particular, for every configuration $\xx$, we choose either $t \in [0,1/N]$ or $t \in [-1/N, 0]$ (i.e. matching the sign of $\<\xx, \bm 1\>$) such that 
		\begin{equation}
				\rho_N(\xx) \leq \rho_N(\xx - t\bm 1) \exp \l( t^2 \frac{N^2\b}{4} \r) \lesssim \rho_N(\xx - t \bm 1).
		\end{equation}
		Then, for example if $\<\xx , \bm 1\> > 0$ we write
		\begin{align*}
				\rho_N(\xx) = N \int_{0}^{1/N} \rho_N(\xx) dt \lesssim N \int_0^{1/N} \rho_N(\xx - t\bm 1) dt \leq N \int_{-1/N}^{1/N} \rho_N(\xx - t \bm 1) dt.
		\end{align*}
		The other case $\<\xx, \bm 1\> \leq 0$ is similar. 
\end{proof}

\section{Elementary estimates}

\begin{lemma}\label{l:gron-x2}
		Fix a $T > 0$. Let $X_t, a_t, b_t$ be continuous and nonnegative, satisfying
		\begin{equation}
				X_t \leq a_t + \int_0^t b_u X_u^2 du, \qquad \t{ for all } 0 \leq t \leq T.
		\end{equation}
		Then, denoting $A_T := \sup_{0 \leq t \leq T}$, it holds
		\begin{equation}\label{e:gron-x2}
				X_t \leq \frac{A_T}{1 - A_T \int_0^t b_u du }.
		\end{equation}
\end{lemma}

\begin{proof}
		Define
		\begin{equation}
				Y_t = \l(\sup_{0 < t < T} a_t \r) +  \int_0^t b_u X_u^2 du.
		\end{equation}
		Then $X_t \leq Y_t$ and furthermore
		\begin{equation}
				Y_t' = b_t X_t^2 \leq b_t Y_t^2.
		\end{equation}
		Thus, 
		\begin{equation}
				\frac{d}{dt} \l( - \frac{1}{Y_t} \r) \leq b_t,
		\end{equation}
		which implies
		\begin{equation}
				\frac{1}{Y_t} \geq \frac{1}{\sup_{0 < t < T} a_t} - \int_0^t b_u du,
		\end{equation}
		and (\ref{e:gron-x2}) follows.
\end{proof}

\begin{lemma}\label{l:log-div-1/Z}
		Let $Z \sim \Nor(\mu, \si^2)$ and $\h > 0$. Then it holds
		\begin{equation}\label{e:1/Z+h}
				\E \l[ \frac{1}{|Z + \i \h|} \r] \lesssim \frac{1}{\si} \l( 1 + \log \l( \frac{\si}{\h} \r) \r).
		\end{equation}
		Moreover, for a jointly Gaussian random vector,
		\begin{equation}
				\begin{bmatrix}
						Z_1	\\ Z_2	
				\end{bmatrix} \sim \Nor \l( 0, \begin{bmatrix}
				\si_1^2 & \d \\ \d & \si_2^2
				\end{bmatrix} \r),
		\end{equation}
		we have the estimate,
		\begin{equation}\label{e:1/Z+h-2}
				\E \l[ \frac{1}{|Z_1 + \i \h| |Z_2 + \i\h|} \r] \lesssim \frac{1}{\sqrt{D}} \min_{k\in\{1, 2\}} \l\{ \l( 1 + \log \l( \frac{\si_k}{\h} \r) \r) \l( 1 + \log \l( \frac{\sqrt{D}}{\h\si_k} \r) \r) \r\},
		\end{equation}
		where we denoted $D = \det \begin{bmatrix} 
				\si_1^2 & \d \\ \d & \si_2^2
		\end{bmatrix}$ the determinant and assume that $D > 0$.
\end{lemma}

\begin{proof}
		First note that it suffices to prove (\ref{e:1/Z+h}) for when $\mu = 0$. Now, if $\si < \h$ we have
		\begin{equation}
				\frac{1}{|Z + \i\h|} \leq \frac{1}{\h} \leq \frac{1}{\si}.
		\end{equation}
		Otherwise, if $\si > \h$ we have
		\begin{align*}
				\E \l[ \frac{1}{|Z + \i\h|} \r] &\lesssim \frac{1}{\si }\l[ \int_0^\si \frac{1}{\sqrt{x^2 + \h^2}} dx + \int_\si^\infty \frac{1}{x} e^{- \frac{x^2}{2\si^2}} dx \r]  \\ 
												&\lesssim \frac{1}{\si} \l[ \l.\log \l( x + \sqrt{x^2 + \h^2} \r)\r|_0^\si + \int_{1}^\infty x e^{-x^2/2} dx \r] \\
												&\lesssim \frac{1}{\si} \l[ \log(\si/\h) + 1 \r].
		\end{align*}
		Now, to show (\ref{e:1/Z+h-2}) we condition on $Z_1$ and apply (\ref{e:1/Z+h}). First, recall that for jointly Gaussian random variables,
		\begin{equation}
				Z_2 | Z_1 \sim \Nor \l( \frac{\d}{\si_1^2} Z_1, \frac{D}{\si_1^2} \r),
		\end{equation}
		where $D := \si_1^2 \si_2^2 - \d^2$ is the determinant. Then applying (\ref{e:1/Z+h}) we can write
		\begin{align*}
				\E \l[ \frac{1}{|Z_1 + \i\h| |Z_2 + \i\h|} \r] &= \E_{Z_1} \l[ \frac{1}{|Z_1 + \i\h|} \E_{Z_2 | Z_1} \l[ \frac{1}{|Z_2 + \i\h|} \r] \r] \\
															   &\lesssim \E_{Z_1} \l[ \frac{1}{|Z_1 + \i\h|}  \r] \l( \frac{\si_1}{\sqrt{D}} \l( 1 + \log \l( \frac{\sqrt{D}}{\h\si_1} \r) \r) \r) \\
															   &= \frac{1}{\sqrt{D}} \l( 1 + \log \l( \frac{\si_1}{\h} \r) \r) \l( 1 + \log \l( \frac{\sqrt{D}}{\h \si_1} \r) \r).
		\end{align*}
		Repeating the argument with $Z_2$ instead of $Z_1$ yields (\ref{e:1/Z+h-2}).
\end{proof}



\end{document}